\documentclass[12pt,a4paper]{amsart}
\usepackage{amssymb,latexsym,amsmath,amscd,graphicx,url,color}
\let\oldlabel=\label
\def\prellabel{\marginparsep=1em\marginparwidth=44pt
    \def\label##1{\oldlabel{##1}\ifmmode\else\ifinner\else
         \marginpar{{\footnotesize\ \\ \tt
                    ##1}}\fi\fi}}

\theoremstyle{plain}
\newtheorem*{mainthm}{Main Theorem}
\newtheorem{thm}{Theorem}[section]
\newtheorem{prop}[thm]{Proposition}
\newtheorem{cor}[thm]{Corollary}
\newtheorem{lemma}[thm]{Lemma}

\theoremstyle{definition}
\newtheorem{defn}[thm]{Definition}
\newtheorem{rmk}[thm]{Remark}
\newtheorem{rmks}[thm]{Remarks}

\newcommand{\NN}{{\mathbb N}}
\newcommand{\PP}{{\mathbb P}}
\newcommand{\QQ}{{\mathbb Q}}

\newcommand{\ZZ}{{\mathbb Z}}

\DeclareMathOperator{\CS}{CS}
\DeclareMathOperator{\HS}{HS}
\DeclareMathOperator{\GL}{GL}
\DeclareMathOperator{\gin}{gin}
\DeclareMathOperator{\inid}{in}

\DeclareMathOperator{\pol}{pol}

\DeclareMathOperator{\Ker}{Ker}
\DeclareMathOperator{\projdim}{projdim}
\DeclareMathOperator{\reg}{reg}

\title{Universal Gr\"obner bases and Cartwright-Sturmfels ideals}
\author{A. Conca}
\address{Dipartimento di Matematica, 
Universit\`a di Genova, Via Dodecaneso 35, 
I-16146 Genova, Italy}
\email{conca@dima.unige.it}

\author{E. De Negri}
\address{Dipartimento di Matematica, 
Universit\`a di Genova, Via Dodecaneso 35, 
I-16146 Genova, Italy}
\email{denegri@dima.unige.it}

\author{E. Gorla}
\address{Institut de Math\'ematiques, Universit\'e de Neuch\^atel, Rue Emile-Argand 11, CH-2000
  Neuch\^atel, Switzerland}  
\email{elisa.gorla@unine.ch}

\thanks{The first and second authors were partially supported by INdAM-GNSAGA. The third author was partially supported by the Swiss National Science Foundation under grant no. 200021\_150207.}


\begin{document}

\begin{abstract}
We describe the universal Gr\"obner basis of the ideal of maximal minors and the ideal of $2$-minors of a multigraded matrix of linear forms. Our results imply that the ideals are radical and provide bounds on the regularity. In particular, the ideals of maximal minors have linear resolutions. Our main theoretical contribution consists of introducing two new classes of ideals named after Cartwright and Sturmfels, and proving that they are closed under multigraded hyperplane sections. The gins of the ideals that we study enjoy special properties.
\end{abstract}

\maketitle

\section*{Introduction}

In this paper we study the universal Gr\"obner bases of the ideals of maximal minors and the ideals of $2$-minors of multigraded matrices of linear forms.
Our results might be seen as the generalization of the Bernstein-Sturmfels-Zelevinsky Theorem \cite{BZ,SZ} (asserting that maximal minors of a matrix of variables form a universal Gr\"obner basis) and of the result of Sturmfels \cite{S} and Villarreal \cite{V} (asserting that the cycles of the complete bipartite graph give rise to the universal Gr\"obner basis of the ideal of $2$-minors of the matrix of variables). Intermediate results in this direction have been obtained in \cite{AST,B,CS,C,CDG,K}.

In order to study universal Gr\"obner bases, we introduce two families of multigraded ideals, which we call Cartwright-Sturmfels and Cartwright-Sturmfels$^*$. Both families are characterized by properties of their multigraded generic initial ideals: Cartwright-Sturmfels ideals have radical generic initial ideals, while the generic initial ideal of a Cartwright-Sturmfels$^*$ ideal has a system of generators that involves only one variable per multidegree. We prove that the generic initial ideal of a Cartwright-Sturmfels or Cartwright-Sturmfels$^*$ ideal is independent of the choice of the term order. 

It turns out that Cartwright-Sturmfels ideals are radical, and that every multigraded minimal system of generators of a Cartwright-Sturmfels$^*$ ideal is a universal Gr\"obner basis. In addition, the Castelnuovo-Mumford regularity of $\ZZ^v$-graded Cartwright-Sturmfels ideals and the projective dimension of $\ZZ^v$-graded Cartwright-Sturmfels$^*$ ideals are bounded by $v$. We also prove that squarefree monomial ideals in the two families are Alexander-dual to each other, and that the generic initial ideals of Alexander duals are obtained from each other via Alexander duality and polarization. 

Finally, we prove that the family of Cartwright-Sturmfels ideals and the family of Cartwright-Sturmfels$^*$ ideals are closed with respect to several natural operations on ideals. For example, both families are closed with respect to going modulo $L$ and to taking the colon by $L$, for any multigraded linear form $L$. As a consequence of our results, we obtain that
the ideals of maximal minors and the ideals of $2$-minors of multigraded matrices of linear forms are Cartwright-Sturmfels. Moreover, the ideals of maximal minors of a column-graded matrix of multigraded linear forms is Cartwright-Sturmfels$^*$ as well. This allows us to derive the desired results about universal Gr\"obner bases.

\bigskip

Let $S$ be  a polynomial ring  over a field $K$ with a standard  $\ZZ^v$-graded structure. By this we mean that the degree of every indeterminate  of $S$ is an element of the canonical basis $\{e_1,\dots,e_v\}$ of $\ZZ^v$. 
Let $A=(a_{ij})$ be a $m\times n$ matrix with entries in $S$ and denote by $I_t(A)$ the ideal of $t$-minors of $A$. We say that the matrix $A$ is column-graded if  $n\leq v$ and $a_{ij}=0$ or $\deg(a_{ij})=e_j\in \ZZ^v$ for every $i,j$. We say that $A$ is row-graded if  $m\leq v$ and $a_{ij}=0$ or $\deg(a_{ij})=e_i\in \ZZ^v$ for every $i,j$.
 
\begin{mainthm}
Assume $A$ is column graded  or row graded of size $m\times n$ with $m\leq n$. Then:  
\begin{enumerate}
\item $I_m(A)$ is radical and has a linear resolution. Moreover, every initial ideal of $I_m(A)$  is radical and has a linear resolution. 
 \item  In the column graded case the maximal minors of $A$ form a universal Gr\"obner basis  of $I_m(A)$.  In the row graded case $I_m(A)$ has  a universal Gr\"obner basis  of  elements of degree equal to $\sum_{i=1}^m e_i  \in \ZZ^m$.
\item $I_2(A)$ is radical. Moreover, every initial ideal of $I_2(A)$   is radical. 
\item  In the column graded case  $I_2(A)$ has  a universal Gr\"obner basis  of elements of degree $\leq \sum_{i=1}^n e_i$.  In the row graded case  $I_2(A)$ has  a universal Gr\"obner basis  of elements of degree $\leq \sum_{i=1}^m e_i$. \end{enumerate} 
\end{mainthm}
 


This corresponds to the following geometric picture: for $i=1,\ldots,v$ fix positive integers $n_i$ and let $x_i=(x_{i1},\ldots,x_{in_v})^t$ be column vectors whose entries are distinct indeterminates. Let $A_i\in K^{m\times n_i}$ be matrices with entries in $K$, let $A$ be the matrix whose columns are $A_1x_1,\ldots,A_vx_v$. Then $A$ is a column graded matrix of size $m\times v$, and it defines a map $$\varphi:\PP^{n_1-1}\times\ldots\times\PP^{n_v-1}\longrightarrow(\PP^{m-1})^v.$$
The ideal $I_2(A)$ defines the inverse image via $\varphi$ of the diagonal $\{(P,\ldots,P) \ : \ P\in\PP^{m-1}\}\subset(\PP^{m-1})^v$, while $I_m(A)$ defines the locus of points of $\PP^{n_1-1}\times\ldots\times\PP^{n_v-1}$ whose images lie on a common hyperplane $H\subset\PP^{m-1}$.
A row graded $A$ of size $m\times n$ can be constructed similarly, and has a similar geometric interpretation.

 \section{Cartwright-Sturmfels ideals}
 
Given $v\in \NN$,  let $S$ be a polynomial ring over an  infinite field $K$ with a standard  $\ZZ^v$-graded structure. 
As said above, by this we mean that the degree of every indeterminate of $S$ is an element of the canonical basis $\{e_1,\dots,e_v\}$ of $\ZZ^v$. 
For $i=1,\dots,v$ let $n_i$ be the number of indeterminates of $S$ of degree $e_i$.  
We denote them by $x_{i1},\dots, x_{in_i}$. We assume that $n_i>0$ for all $i$. 
 
For a $\ZZ^v$-graded $S$-module $M$ we denote by $M_a$ the homogeneous component of $M$ of degree $a\in \ZZ^v$. 

\begin{defn}
Let $M$ be a finitely generated, $\ZZ^v$-graded $S$-module. 
The {\em $\ZZ^v$-graded Hilbert series} of $M$ is
 $$\HS(M,y_1,\dots,y_v)=\sum_{a\in \ZZ^v} (\dim M_a)  y^a\in
 \QQ[[y_1,\dots,y_v]][y_1^{-1},\dots, y_v^{-1}].$$ 
 If $v$ is clear from the context, we simply denote it by $\HS(M,y)$. 
\end{defn} 

The group $G=\GL_{n_1}(K)\times \cdots \times \GL_{n_v}(K)$ acts on
$S$ as the  group of $\ZZ^v$-graded $K$-algebras automorphisms.  Let
$B=B_{n_1}(K)\times \cdots  \times B_{n_v}(K)$ be the {\em Borel
  subgroup} of $G$, consisting of the upper triangular invertible
matrices. 
\begin{defn}
An ideal $I\subset S$ is {\em Borel fixed} if $g(I)=I$ for every $g\in
B$. 
\end{defn}

Borel fixed ideals are monomial ideals that can be characterized by
means of exchange properties, as explained in \cite[Thm. 15.23]{E} for
standard $\ZZ$-graded polynomial rings and in  \cite[Sect.1]{ACD} in the standard $\ZZ^2$-graded case. 
Explicitly,  an ideal  $I$ of $S$  is Borel fixed with respect to the $\ZZ^v$-graded structure if it is generated by monomials and for every monomial generator 
$u$ of $I$ and every variable $x_{ij}$ appearing in $u$ with exponent, say, $c$ one has that $(x_{ik}/x_{ij})^du\in I$ for every $k<j$ and every 
$0\leq d\leq c$ such that $\binom{c}{d}\neq 0$ in the field  $K$. 

\begin{rmk} 
Unless otherwise stated,  we consider only term orders such that $x_{ij}>x_{ik}$ if $1\leq j<k\leq n_i$. 
\end{rmk} 

Given a term order $\tau$ and a $\ZZ^v$-graded homogeneous ideal $I$ of $S$, one can consider its {\em $\ZZ^v$-graded generic initial ideal} $\gin_{\tau}(I):=\inid_\tau(g(I))$, where $g$ is a general element in $G$. As in the $\ZZ$-graded setting, $\ZZ^v$-graded generic initial ideals are Borel fixed and can be obtained as $\inid_{\tau}(b(I))$ for a general $b\in B$. 

Inspired by the work of  Cartwright and Sturmfels \cite{CS} we introduce two classes of ideals that will play an important role in the proof of the Main Theorem.  

Let  $T=K[x_{11},x_{21},\dots, x_{v1}]\subset S$ with the $\ZZ^v$-graded structure induced by that of $S$. Note that a $\ZZ^v$-graded homogeneous ideal of $T$ is nothing but a monomial ideal of $T$.  Similarly, a $\ZZ^v$-graded ideal of $S$ which is extended from $T$ is an ideal of $S$ which is generated by monomials in the variables $x_{11},x_{21},\dots, x_{v1}$.

\begin{defn}
Let $I$ be a $\ZZ^v$-graded ideal of $S$. We say that $I$ is a Cartwright-Sturmfels ideal if there exists a radical Borel fixed ideal $J$ of $S$ such that $\HS(I,y)=\HS(J,y)$. We say that $I$ is a Cartwright-Sturmfels$^*$ ideal if there exists a $\ZZ^v$-graded ideal $J$ of $S$ extended from $T$ such that $\HS(I,y)=\HS(J,y)$. 
\end{defn}

The set of Cartwright-Sturmfels ideals of $S$ is denoted by $\CS(S)$, or simply by $\CS$ when $S$ is clear from the context. Similarly $\CS^*(S)$, or simply 
$\CS^*$, denotes the set of the Cartwright-Sturmfels$^*$ ideals of $S$. 

The classes $\CS$ and $\CS^*$ are in a sense dual to each other. In fact, in Theorem \ref{dual} we show that if $I$ is a squarefree monomial ideal, then $I$ belongs to $\CS$ if and only if its Alexander dual $I^*$ belongs to $\CS^*$.

\begin{rmks}\label{easy}
It follows from the definitions that the families $\CS$ and $\CS^*$ are closed under $\ZZ^v$-graded coordinate changes, i.e. under the action of the group $G$. Moreover, $I\in\CS$ (respectively in $\CS^*$) if and only if $\inid_{\tau}(I)\in \CS$ (respectively in $\CS^*$), for some term order $\tau$.
\end{rmks} 

The next proposition was already proved in~\cite{CDG}.

\begin{prop} 
Let $I\in \CS$ and let $J$ be a radical Borel fixed ideal such that $\HS(I,y)=\HS(J,y)$. Then for every term order $\tau$ one has $\gin_\tau(I)=J$. 
In particular, $I$ is radical and generated by elements of degree $\leq \sum_{i=1}^v e_i$. Moreover one has the following characterization of 
Cartwright-Sturmfels ideals: 
$$\CS=\{ I : I \mbox{ is }  \ZZ^v\mbox{-graded and $\gin_{\tau}(I)$ is radical for some (equivalently all) term order $\tau$}\}.$$
\end{prop} 
\begin{proof} The first statement follows from  \cite[Theorem~2.5]{CDG}  applied to $J$ and $\gin_\tau(I)$. The rest is an easy consequence. 
\end{proof}   

\begin{prop}\label{aru1}
Let $S'=S[z_1,\dots,z_u]$ be a polynomial extension of $S$ with a standard $\ZZ^v$-graded structure extending that of $S$.  
Let $I$ be a $\ZZ^v$-graded ideal of $S$. The following are equivalent: 
\begin{itemize}
\item[(1)] $I\in \CS(S)$, 
\item[(2)] $IS'\in  \CS(S')$, 
\item[(3)] $IS'+(z_1,\dots,z_u)\in  \CS(S')$.
\end{itemize} 
\end{prop} 

\begin{proof} 
Let $\tau$ be a term order on $S$, let $\sigma$ be a term order on $S'$ which extends $\tau$ and such that 
$x>_\sigma z_i$ for every $i$ and every variable of $x$ of $S$. One first observes that 
$$\gin_\tau(I)S'=\gin_\sigma(IS'),$$
since the generic initial ideal can be obtained as $\inid(b(I))$, for a general element $b\in B$. 
The equality implies immediately that (1) and (2) are equivalent.  

To prove that (2) and  (3) are equivalent, one first notices that  
$$\HS(IS'+(z),y)=\HS(\gin_\tau(I)S'+(z),y)$$
because $z=\{z_1,\dots,z_u\}$ is a regular sequence mod $IS'$  and mod $\gin_\tau(I)S'$. 

Assuming (2) we have that $\gin_\tau(I)$ is radical and Borel fixed, hence $\gin_\tau(I)S'+(z)$ is radical and Borel fixed provided that we order the variables of $S'$ so that the $z_i$'s are the first in their multidegrees.
Hence $IS'+(z)$ has the Hilbert series of a radical and Borel fixed ideal, so $IS'+(z)\in \CS(S')$. 
 
Assuming (3) we know that all the ideals with the Hilbert series of $IS'+(z)$ are in $\CS(S')$ and in particular are radical. This implies that $\gin_\tau(I)S'+(z)$ is radical, hence $\gin_\tau(I)$ is radical. 
 \end{proof} 
 
 
\begin{rmk} Let $S=K[x_{ij} : 1\leq i\leq  m \mbox{ and } 1\leq j \leq  n]$ with the grading induced by $\deg(x_{ij})=e_i\in \ZZ^m$. Denote by $I_2$ the ideal of $2$-minors of $(x_{ij})$. 
In \cite{CS}  Cartwright and Sturmfels proved that every $\ZZ^m$-graded ideal $J$ with Hilbert series equal to that of $I_2$ is radical. Their argument relies on the fact that $\gin(I_2)$ is radical, a result proved by Conca in \cite{C}. This circle of  ideas lead us to consider the family of multigraded ideals with a radical gin, and to name them after Cartwright and Sturmfels.
\end{rmk}

We now turn our attention to the ideals in $\CS^*$. We start by establishing some properties, which will be essential in the sequel.

\begin{prop}
\label{arus}
Let $J$ be a $\ZZ^v$-graded ideal such that $J\in \CS^*(S)$ and let $C=( x_{11}^{a_1}\cdots x_{v1}^{a_v} : a\in \ZZ^v \mbox{ and } J_a\neq 0)$. Then: 
 \begin{enumerate}
\item $C=\gin_\tau(J)$ for every term order $\tau$ and $C$ is the only Borel fixed ideal with the same Hilbert series as $J$. 
\item For general $\lambda_{ij}\in K$ the sequence $\Lambda=\{x_{ij}-\lambda_{ij}x_{i1}\mid 1\leq i\leq v,\; 2\leq j\leq n_i\}$ is   $S/J$-regular   and $S/C$-regular. 
\item The ideals $J$ and $C$ have the same $\ZZ^v$-graded Betti numbers. In particular they have the same $\ZZ$-graded Betti numbers and 
$$\projdim(S/J)\leq v.$$
\item The multidegrees of any minimal system of $\ZZ^v$-homogeneous generators of $J$ are incomparable in $\ZZ^v$. 
\end{enumerate} 
\end{prop} 

\begin{proof} 
(1) By assumption there exists an ideal $C'$ generated by monomials in $x_{11},\dots, x_{v1}$ with $\HS(C',y)=\HS(J,y)$. 
We claim that $C=C'$. In fact, let $a\in \ZZ^v$ such that $u=x_{11}^{a_1}\cdots x_{v1}^{a_v}$  is  a monomial generator of $C'$. 
Then $0\neq \dim (C')_a=\dim J_a$, hence $u$ is in $C$. Conversely, let $a\in \ZZ^v$ such that $J_a\neq 0$. Hence $(C')_a\neq 0$, so there is a generator $x_{11}^{b_1}\dots x_{v1}^{b_v}$ of $C'$ such that $b\leq a$ coefficientwise. In particular, $x_{11}^{a_1}\cdots x_{v1}^{a_v}\in C'$. 
Next we show that $C=\gin_\tau(J)$. Since $\HS(C,y)=\HS(\gin_\tau(J),y)$, it suffices to show that $C\subseteq \gin_\tau(J)$.  If $a\in \ZZ^v$ is such that $C_a\neq 0$, then $\gin_\tau(J)_a\neq 0$. Since $\gin_\tau(J)$ is a Borel fixed ideal which contains a monomial of degree $a$, it contains $x_{11}^{a_1}\cdots x_{v1}^{a_v}$. The same argument shows that $C$ is the only Borel fixed ideal with the Hilbert series of $J$. 

(2) We know from (1) that $\{x_{ij}\mid 1\leq i\leq v,\; 2\leq j\leq n_i\}$ is a regular sequence modulo $\gin_{\tau}(J)$, hence also modulo $g(J)$ for a generic $g\in G$. Let  $L_{ij}=g^{-1}(x_{ij})$. By construction $\{L_{ij}\mid 1\leq i\leq v,\; 2\leq j\leq n_i\}$ is $S/J$-regular. Since $g$ is  generic, applying Gaussian elimination to the $L_{ij}$'s produces a system of generators of the same ideal of the form $x_{ij}-\lambda_{ij}x_{i1}$, for general $\lambda_{ij}$. Since the property of being a regular sequence just depends on the ideal that the elements  generate, it follows that the sequence $\Lambda$ is $S/J$-regular. The assertion for $C$ is obvious. 

Finally, (3) follows from (2) since by construction $J+(\Lambda)=C+(\Lambda)$, and (4) follows from (3) because the assertion is obviously true for $C$. 
\end{proof} 
 
The following characterization of Cartwright-Sturmfels* ideals is a simple consequence of Proposition~\ref{arus}.

\begin{cor}\label{easy2} We have: 
$$\CS^*=\{ I : I \mbox{ is }  \ZZ^v\mbox{-graded and $\gin_{\tau}(I)$ is extended from $T$ for some (equivalently all) $\tau$} \}.$$
\end{cor} 

Moreover we have: 
\begin{cor}\label{easy3}
Let $J\subset S$ be a monomial ideal. The following are equivalent:  
\begin{enumerate} 
\item $J\in \CS^*$, 
\item $\Gamma=\{x_{ij}-x_{i1}\mid 1\leq i\leq v,\;  2\leq j\leq n_i\}$ is a $S/J$-regular sequence.
\end{enumerate}
\end{cor}
\begin{proof}
Assuming (1), by Proposition~\ref{arus} (4) there exist $\lambda_{ij}\in K^*$ such that $\Lambda=\{x_{ij}-\lambda_{ij}x_{i1}\mid 1\leq i\leq v,\;  2\leq j\leq n_i\}$ is an $S/J$-regular sequence. Since $J$ is monomial, $J+(\Lambda)=J+(\Gamma)$, so $\Gamma$ is $S/J$-regular.
Assuming  (2) we have that $J+(\Gamma)=C+(\Gamma)$  for some ideal $C$ extended from $T$. 
Then $\Gamma$ is $S/C$-regular, hence $\HS(J,y)=\HS(C,y)$. 
\end{proof} 

We stress the following  property of ideals in $\CS^*$.

\begin{prop}
Let $J\in\CS^*$. 
Then any minimal system of $\ZZ^v$-graded generators of $J$ is a universal Gr\"obner basis of $J$.
\end{prop}

\begin{proof}
Let $\tau$ be a term order and let $H=\inid_\tau(I)$. Then $H\in \CS^*$. By Proposition~\ref{arus} (1), (3), and (4) $H$ and $J$ have the same graded Betti numbers and their minimal generators have incomparable degrees. It follows that any minimal system of $\ZZ^v$-graded generators of $J$ is a Gr\"obner basis with respect to $\tau$. 
\end{proof}

We have seen that  the generic initial ideal of an ideal $I$ which is either in $\CS$ or in $\CS^*$ is independent of the term order. Hence we simply denote it by $\gin(I)$. If $I$ is a squarefree monomial ideal, we denote by $I^*$ its Alexander dual. If $I$ is a monomial ideal, we denote by $\pol(I)$ its polarization.  
An application of the Alexander inversion formula (\cite[Theorem 5.14]{MS}) yields the following:  

\begin{lemma}\label{MS-Alex}
Let $I,J\subset S$ be squarefree monomial ideals with the same $\ZZ^v$-graded Hilbert series. Then also $I^*$ and $J^*$ have the same $\ZZ^v$-graded Hilbert.
\end{lemma} 
  
\begin{thm}\label{dual}
Let $I\subset S$ be a squarefree monomial ideal. Then $I\in\CS$ if and
only if $I^*\in\CS^*$. Moreover, if $I\in\CS$, then $\gin(I)^*=\pol(\gin(I^*)).$
\end{thm}

\begin{proof}
Let $I\in\CS$ and $B=\gin(I)$. Since $I$ and $B$ have the same $\ZZ^v$-graded Hilbert series,  the same holds for their Alexander duals $I^*$ and $B^*$ by Lemma~\ref{MS-Alex}. Since being in $\CS^*$ depends only on the Hilbert series, it suffices to show that $B^*\in \CS^*$. 
Since $B$ is squarefree there exist  elements $a_1,\dots, a_t\in \NN^v$ such that $B=\cap_{k=1}^t  P_{a_k}$,
where $P_{a_k} =(x_{ij} : i=1,\dots  v \mbox{ and } 1\leq j\leq  a_{ki} )$.
Hence $B^*=(\prod_{i=1}^v \prod_{j=1}^{a_{ki}} x_{ij}  : k=1,\dots,t)=\pol(C)$, where 
$$C=(x_{11}^{a_{i1}}\cdots x_{v1}^{a_{iv}}\mid 1\leq i\leq t).$$
In particular $\HS(B^*,y)=\HS(C,y)$, so $B^*\in\CS^*$. 
 
Conversely, assume that $J\in\CS^*$ is a squarefree monomial ideal, let $C=\gin(J)$. 
Then $C$  is generated by monomials in the variables $x_{11},\dots, x_{v1}$ and $\HS(J,y)=\HS(C,y)$. 
The ideal $\pol(C)$ is squarefree with $\HS(J,y)=\HS(\pol(C),y)$. By construction 
$\pol(C)^*$ is radical and Borel fixed, and by Lemma~\ref{MS-Alex} we have $\HS(J^*,y)=\HS(\pol(C)^*,y)$. 
Therefore $J^*\in\CS$ and $\gin(J^*)=\pol(\gin(J))^*$.
\end{proof}


The following is a simple consequence of Theorem~\ref{dual}.

\begin{cor}
\label{reg}
Let $I$ be a $\ZZ^v$-graded ideal of $S$. If $I\in\CS(S)$ then $\reg(I)\leq v$. 
\end{cor}

\begin{proof}
Let $\sigma$ be a term order on $S$.
Since $\reg(I)\leq \reg (\inid_{\sigma}(I))$ and $\inid_{\sigma}(I)\in \CS$, we may assume without loss of generality that $I$ is monomial. 
By Terai's Theorem \cite[5.59]{MS} one has $\reg(I)=\projdim(S/I^*)$. Since $I^*\in \CS^*$ we have $\projdim(S/I^*)\le v$ by Proposition~\ref{arus} (3). 
\end{proof} 

Finally, we prove that $\CS$ and $\CS^*$ are closed under certain natural operations.
 
\begin{thm}\label{closures}
Let $L$ be a $\ZZ^v$-graded  linear form of $S$. In the following  $S/(L)$ is identified with a polynomial ring with the induced $\ZZ^v$-graded structure.
Let $U_i\subseteq S_{e_i}$ be vector subspaces for all $i=1,\ldots,v$ and let $R=K[U_1,\ldots,U_v]$ be the $\ZZ^v$-graded polynomial subring of $S$ that they generate. 
Then:
\begin{enumerate}
\item If $I\in \CS^*(S)$, then $I+(L)/(L)\in \CS^*(S/(L))$. 
\item If $I\in \CS^*(S)$, then  $I:L\in \CS^*(S)$.  
\item If $I\in \CS^*(S)$, then  $I\cap (L) \in \CS^*(S)$.  
\item If $I\in \CS(S)$, then $I:L\in \CS(S)$. 
\item If $I\in \CS(S)$, then  $I+(L)\in \CS(S)$ and $I+(L)/(L)\in \CS(S/(L))$.
\item If $I\in \CS(S)$, then $I\cap R\in \CS(R)$. 
\end{enumerate}
\end{thm}

\begin{proof} (1) Being in $\CS^*$ just depends on the Hilbert series. Hence, by changing coordinates and passing to the initial ideal with respect to a suitable revlex order, we may assume without loss of generality that $I$ is monomial and $L=x_{1n_1}$. 
Identifying $S/(x_{1n_1})$ with the subring $R$ of $S$ in the variables different from $x_{1n_1}$, the ideal $I+(x_{1n_1})/(x_{1n_1})$ may be identified with the ideal $J$ of $R$ generated by the monomials of $I$ which are  not divisible by $x_{1n_1}$. 
The quotient $R/J$ is an algebra retract of $S/I$, in particular it is a direct summand of $S/I$. In other words we have a decomposition $S/I=R/J\oplus M$ as  $R$-modules, where $M$ is the kernel of the projection $S/I\to R/J$. The sequence $$\Lambda=\{ x_{ij}-x_{i1} | \ 1\leq i\leq v, 2\leq j\leq n_i,  (i,j)\neq (1,n_1) \}$$ of elements of $R$ is $S/I$-regular by Corollary~\ref{easy3}. Hence it is also regular on the $R$-direct summands of $S/I$. In particular, $\Lambda$ is $R/J$-regular. Then $J\in \CS^*(R)$ by Corollary~\ref{easy3}. 

(2) As in the proof of (1) we may assume without loss of generality that $I$ is monomial and $L=x_{1n_1}$. We have a short exact sequence: 
$$0\to S/I:L(-1)\to S/I\to S/I+(L)\to 0$$
By (1) and Corollary~\ref{easy3}\; $\Lambda$ is  $S/I+(L)$-regular, and $\Gamma=\{ x_{ij}-x_{i1} | \ 1\leq i\leq v, 2\leq j\leq n_i\}$ is $S/I$-regular. Now it is a simple exercise on Koszul homology to prove that $\Gamma$ is also $S/(I:L)$-regular. By Corollary~\ref{easy3} we conclude that $I:L\in \CS^*(S)$. 

(3) As in (2) one may assume without loss of generality that $I$ is monomial and $L=x_{1n_1}$. Using the short exact sequence
 $$0\to S/I\cap (L) \to S/I\oplus  S/(L) \to S/I+(L)\to 0$$
 one can now prove that $I\cap (L)\in \CS^*(S)$. 
 
 (4) Since being in $\CS$ only depends on the Hilbert series, we may assume without loss of generality that $I$ is monomial and $L=x_{11}$. Denote by $R$ the polynomial subring of $S$ generated by the variables different from $x_{11}$. By Theorem~\ref{dual} the Alexander dual $I^*$ of $I$ is in $\CS^*(S)$. 
 Let $J$ be the ideal of $R$ generated by the monomials in $I^*$ that are not divisible by $x_{11}$. Since $J$ may be identified with $I^*+(x_{11})/(x_{11})\subset S/(x_{11})$, then $J\in \CS^*(R)$ by (1). Hence $JS\in \CS^*(S)$ by Proposition~\ref{aru1}. Since $JS=(I:x_{11})^*$, by Theorem~\ref{dual} we conclude that $I:x_{11}\in \CS(S)$. 
   
(5)  As above we may assume without loss of generality that $I$ is monomial and $L=x_{11}$. The Alexander dual of $I+(x_{11})$ is $I^*\cap (x_{11})$ and it belongs to $\CS^*(S)$ by (2). Hence $I+(x_{11})\in \CS(S)$ by Theorem~\ref{dual}. Denote  by $R$ the polynomial subring of $S$ generated by the variables different from $x_{11}$. There exist an ideal $J$ of $R$ such that $JS+(x_{11})= I+(x_{11})$. The gin of $I+(x_{11})$ has the form $(x_{11})+J_1S$, where $J_1$ is a squarefree monomial ideal of $R$ which is Borel fixed in $R$. The ideal $I+(L)/(L)$ may be identified with $J$ that, by construction, has the same Hilbert series as $J_1$. Hence $J\in \CS(R)$. 

(6) Arguing by induction, we may assume that $R$ is obtained form $S$ by removing only one variable, say $x_{11}$. Taking initial ideals with respect to an elimination order, we may assume that $I$ is monomial. Let $J=I\cap R$ and notice that $J$ may be identified with $I+(x_{11})/(x_{11})$. Then $J$ is in $\CS(R)$ by (5). 
\end{proof}

\begin{rmk}  Let $I\in \CS$. Then it follows from Theorem~\ref{closures} (2) that $I:F\in \CS$ where $F$ is a product of $\ZZ^v$-graded linear forms. 
Notice however that $I:F\not\in\CS$ in general, if $F$ is a $\ZZ^v$-graded form. For example, let $S=K[x_{ij}\ | \  1\le i,j\le3]$ with $\deg x_{ij}=e_i$ and 
$I=I_2(X)$ with $$X=\left(\begin{array}{ccc}
x_{11} & x_{12} & x_{13} \\
x_{21} & x_{22} & 0 \\
0 & 0 & x_{33}
\end{array}\right).$$
In Corollary~\ref{6CS} we will prove that $I\in \CS(S)$. Let $F=x_{11}x_{21}x_{32} + x_{13}x_{23}x_{33}$, then
$$I:F=I+(x_{12}x_{13}, x_{11}x_{13}).$$
Notice that $I:F\not\in \CS$, since it has generators of degree $(2,0,0)$, while ideals in $\CS$ are generated in degrees bounded by $\sum e_i$. 
By replacing $I$ with its initial ideal 
$J=(x_{12}x_{21}, x_{13}x_{21}, x_{13}x_{22}, x_{11}x_{33}, x_{12}x_{33}, x_{21}x_{33}, x_{22}x_{33})$ one obtains an example of a monomial ideal in $\CS$ such that $J:F\not\in \CS$. 
\end{rmk} 

 \begin{rmk}  
If $I\in\CS^*$ and $L$ is a $\ZZ^v$-graded linear form, then $I+(L)\not\in\CS^*$ in general. 
For example, $(x_{11},x_{12})\not\in\CS^*$ because ideals in $\CS^*$ have generators with incomparable degrees. 
\end{rmk} 

We are finally ready to prove our main result.

\begin{proof}[Proof of the Main Theorem]  
Assume first that $A$ is row  graded. Let $X=(x_{ij})$ be an $m\times n$ matrix of variables over $K$, and let $R=K[x_{ij}]$ be the polynomial ring 
with standard $\ZZ^m$-grading induced by $\deg(x_{ij})=e_i\in \ZZ^m$. The assignment $x_{ij}\mapsto a_{ij}$ gives rise to a $\ZZ^m$-graded 
$K$-algebras homomorphism $\Phi:R\to S$, whose kernel $J=\Ker \Phi$ is generated by $\ZZ^m$-graded linear forms. 
Hence $\Phi$ induces a $\ZZ^m$-graded $K$-algebra isomorphism: $$R/I_m(X)+J\simeq S'/I_m(A)$$ 
where $S'$ is the $K$-subalgebra of $S$ generated by the entries of $A$. 
By \cite[Theorem 1.1]{CDG} we have that $I_m(X)\in \CS(R)$. It follows from Theorem~\ref{closures} (5) that $I_m(X)+J\in \CS(R)$, hence $I_m(A)\in \CS(S')$. Finally $I_m(A)\in \CS(S)$ by Proposition~\ref{aru1}. Now all the results follow from the general properties of ideals in $\CS$ that we have established. In particular, $\reg I_m(A)\leq m$ follows from Corollary~\ref{reg} and the assertions on the initial ideals follow from Remark~\ref{easy}. The fact that the initial ideals are all generated in multidegree at most $(1,\ldots,1)$ was shown in~\cite[Corollary 3.7]{CDG}.

In the column graded case, we consider the $\ZZ^n$-grading on $R=K[x_{ij}]$ induced by $\deg(x_{ij})=e_j\in \ZZ^n$.
The results for $I_m(A)$ follow from the same approach, with the exception of the statement about the linearity of the resolution which was proved in \cite{CDG}.  

For $I_2(A)$ in the row or column graded case, one applies the same approach as above and uses that $I_2(X)\in \CS(R)$ for a matrix of variables, a result proved in \cite{C}. 
\end{proof} 

We wish to stress the following consequence of the previous proof.

\begin{cor}\label{6CS}
With the notation of the Main Theorem we have that $I_m(A)\in \CS$ and $I_2(A)\in \CS$. Moreover, in the column-graded case $I_m(A)\in \CS\cap \CS^*$. 
\end{cor} 

In addition to the families of determinantal ideals treated in Corollary~\ref{6CS}, we have identified other classes of Cartwright-Sturmfels and Cartwright-Sturmfels$^*$ ideals, 
such as multigraded projective closures of linear spaces (generalizing the ideals studied in~\cite{AB}) and the binomials edge ideals from~\cite{HHHKR}. 
These results, together with a description of the multigraded generic initial ideals of the determinantal ideals above, will be the object of a paper in preparation.

{\bf Acknowledgements:} The authors are grateful to B. Sturmfels and M. Varbaro for useful discussions concerning the material of this papers. Our results are based on extensive computations performed with the computer algebra system CoCoA~\cite{Cocoa}.


\begin{thebibliography}{99}

\bibitem{AST}  C. Aholt, B. Sturmfels, R. Thomas,
{\em A Hilbert schemes in computer vision}
 Canad. J. Math. 65 (2013), no. 5, 961--988.

\bibitem{ACD} A. Aramova, K. Crona, E. De Negri, 
{\em Bigeneric initial ideals, diagonal subalgebras and bigraded Hilbert functions.} 
J. Pure Appl. Algebra 150 (2000), no. 3, 215--235. 

\bibitem{AB} F. Ardila, A. Boocher,
{\em The closure of a linear space in a product of lines},
Journal of Algebraic Combinatorics 43, no. 1, 199--235.

 \bibitem{BZ}
D. Bernstein, A. Zelevinsky,
\emph{Combinatorics of maximal minors. } 
J. Algebraic Combin. 2 (1993), no. 2, 111--121. 

\bibitem{B} A. Boocher, {\em Free resolutions and sparse determinantal ideals}, 
Math. Res. Lett. 19 (2012), no. 4, 805--821.
 
\bibitem{CS} D. Cartwright, B. Sturmfels,
{\em The Hilbert scheme of the diagonal in a product of projective spaces.}
Int. Math. Res. Not. 9 (2010), 1741--1771. 

\bibitem{Cocoa} CoCoA Team,
 \emph{ CoCoA: a system for doing
  Computations in Commutative Algebra},
  Available at http://cocoa.dima.unige.it

\bibitem{C} A. Conca, {\em Linear spaces, transversal polymatroids and ASL domains}.  
J. Algebraic Combin. 25 (2007), no. 1, 25--41. 

\bibitem{CDG} A. Conca, E. De Negri, E. Gorla,
{\em Universal Gr\"obner bases for maximal minors.}
 Int. Math. Res. Not. IMRN 2015, no. 11, 3245--3262.
 
\bibitem{E} D. Eisenbud,
{\em Commutative algebra. With a view toward algebraic geometry.}
Graduate Texts in Mathematics, 150. Springer-Verlag, New York, 1995. 

\bibitem{HHHKR} J. Herzog, T. Hibi, F. Hreinsd\'ottir, T. Kahle, J. Rauh
{\em Binomial edge ideals and conditional independence statements}
Advances in Applied Mathematics 45 (2010), no. 3, 317--333.
 
\bibitem{K}
M. Y. Kalinin,
\emph{Universal and comprehensive Gr\"obner bases of the classical determinantal ideal.} 
 Zap. Nauchn. Sem. S.-Peterburg. Otdel. Mat. Inst. Steklov. (POMI) 373 (2009), Teoriya Predstavlenii, Dinamicheskie Sistemy, Kombinatornye Metody. XVII, 134--143, 348; translation in  J. Math. Sci. (N. Y.) 168 (2010), no. 3, 385--389. 

\bibitem{MS}
E.~Miller, B.~Sturmfels, \emph{Combinatorial commutative algebra.} Graduate Texts in Mathematics, 227. Springer-Verlag, 2005.


\bibitem{SZ}
B. Sturmfels, A. Zelevinsky,
\emph{Maximal minors and their leading terms. } 
Adv. Math. 98 (1993), no. 1, 65-112. 


\bibitem{S}  B.Sturmfels,  
{\em    Gr\"obner Bases and Convex Polytopes},  
Amer. Math. Soc., Providence, RI,1995.

\bibitem{V} R.Villarreal, 
{\em Monomial algebras}. 
Monographs and Textbooks in Pure and Applied Mathematics, 238. Marcel Dekker,
Inc., New York, 2001. 

\end{thebibliography}
\end{document}